\documentclass[12pt]{amsart}

\usepackage[utf8]{inputenc}
\usepackage{amsmath}
\usepackage{amssymb}
\usepackage{amsthm}
\usepackage{graphicx}

\newtheorem{theorem}{Theorem}
\newtheorem{lemma}{Lemma}
\newtheorem{definition}{Definition}
\newtheorem{remark}{Remark}

\newtheorem{corollary}{Corollary}

\title[On computable classes of equidistant sets...]{On computable classes of equidistant sets: multivariate equidistant functions}
\author{\'{A}. Nagy}
\address{Institute of Mathematics, University of Debrecen, H-4002 Debrecen, P. O. Box 400, Hungary}
\email{abris.nagy@science.unideb.hu}
\author{M. Ol\'{a}h}
\address{Institute of Mathematics, University of Debrecen, H-4002 Debrecen, P. O. Box 400, Hungary \newline
\indent ELKH-DE Equations, Functions, Curves and their Applications Research Group}
\email{olah.mark@science.unideb.hu}
\author{M. Stoika}
\address{Department of Mathematics and Informatics, Ferenc Rakoczi II Transcarpathian Hungarian \newline
\indent College of Higher Education, 90200 Beregszász, P.O.Box 33, Transcarpathia, Ukraine}
\email{sztojka.miroszlav@kmf.org.ua}
\author{Cs. Vincze}
\address{Institute of Mathematics, University of Debrecen, H-4002 Debrecen, P. O. Box 400, Hungary}
\email{csvincze@science.unideb.hu}

\begin{document}
\begin{abstract}
An equidistant set in the Euclidean space consists of points having equal distances to both members of a given pair of sets, called focal sets. Having no effective formulas to compute the distance of a point and a set, it is hard to determine the points of an equidistant set in general. Instead of explicit computations we can use computer-assisted methods due to the basic theorem of M. Ponce and P. Santib\'a\~nez about the (Hausdorff) convergence of equidistant sets under the convergence of the focal sets. The authors also give an error estimation process in \cite{PS} to approximate the equidistant points. An alternative way of the approximation is based on finite focal sets as one of the most important computable classes of equidistant sets \cite{VVOFS}, see also \cite{VOL}. Special classes of equidistant sets allow us to approximate the equidistant points in more complicated cases. In what follows we have a hyperplane corresponding to the first order (linear) approximation for one of the focal sets and the second one is considered as the epigraph of a function. This idea results in the construction of equidistant functions. 

In the first part of the paper we prove that the equidistant points having equal distances to the epigraph of a positive-valued continuous function and its domain form the graph of a multivariate function. Therefore such an equidistant set is called a multivariate equidistant function. It is a higher-dimensional generalization of functions in \cite{VVOF} weakening the requirement of convexity as well. We also prove that the equidistant function one of whose focal sets is constituted by the pointwise minima of finitely many positive-valued continuous functions is given by the pointwise minima of the corresponding equidistant functions. In the second part of the paper we consider equidistant functions such that one of the focal sets is the epigraph of a convex function under some smoothness conditions. Independently of the dimension of the space we present a special parameterization for the equidistant points based on the closest point property of the epigraph as a convex set and we give the characterization of the equidistant functions as well. Illustrating how the formulas are working we present an example with a hyperboloid of revolution as one of the focal sets.
\end{abstract}
\maketitle

\section{Introduction: notations and preliminaries}

Let $K\subset \mathbb{R}^n$ be a subset in the Euclidean coordinate space. The distance between a point $x\in \mathbb{R}^n$ and $K$ is measured by the usual infimum formula
$$d(x, K) := \inf  \{ d(x, y) \ | \ y\in K \},$$
where 
$$d(x,y)=|x-y|=\sqrt{(x^1-y^1)^2+\ldots+(x^n - y^n)^2}$$
is the distance coming from the canonical inner product. Let us define the equidistant set of the focal sets $K$ and $L\subset \mathbb{R}^n$ as the set 
$$\{K=L\}:=\{x\in \mathbb{R}^n\ | \ d(x, K)=d(x, L)\}$$
all of whose points have the same distances from both $K$ and $L$. Equidistant sets can be considered as a generalization of conics. They are often called midsets. The systematic investigations have been started by Wilker's and Loveland's fundamental works \cite{Loveland} and \cite{Wilker}. "We find equidistant sets as conventionally
defined frontiers in territorial domain controversies: for instance, the United Nations
Convention on the Law of the Sea (Article 15) establishes that, in absence of any previous agreement, the delimitation of the territorial sea between countries occurs exactly on the median line every point of which is equidistant of the nearest points to each country"; for the citation see \cite{PS}. The points of an equidistant set are difficult to determine in general because there is no effective formula to compute the distance between a point and a set. Special classes of equidistant sets allow us to approximate the equidistant points in more complicated cases. In what follows we have a hyperplane corresponding to the first order (linear) approximation for one of the focal sets and the second one is considered as the epigraph of a function (at least locally), see Figure 1. This idea results in the construction of equidistant functions.

\begin{figure}[h]
\label{fig:01}
\centering
\includegraphics[scale=0.3]{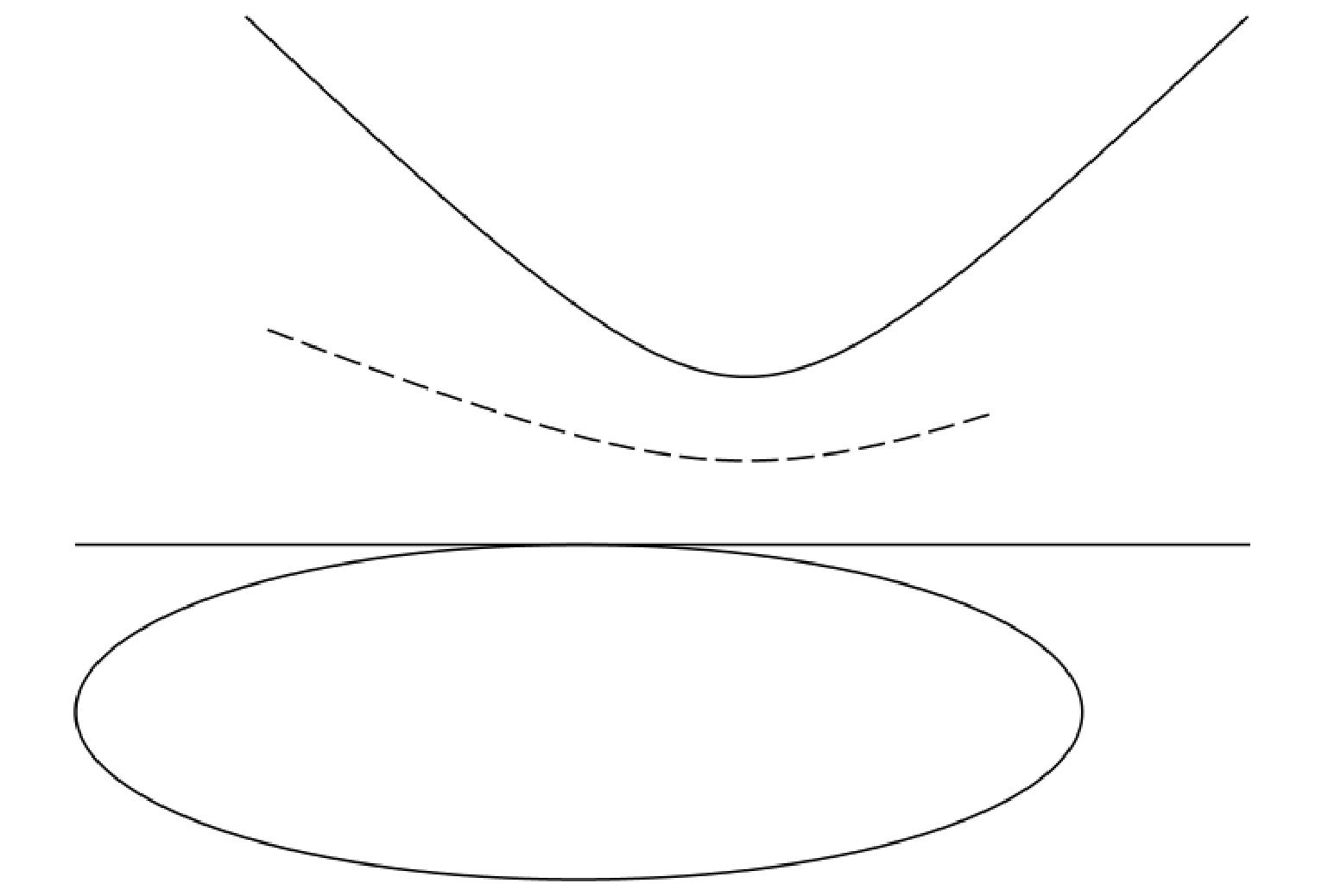}
\caption{The case of a convex function: linear approximation and equidistant points given by equidistant parameterization (see Subsection \ref{parameterization}).}
\end{figure}

In the first part of the paper we give some general observations about the so-called equidistant functions. They are special equidistant sets with the epigraph of a positive-valued continuous function and its domain as focal sets. We prove among others that the minimum operator acts on the family of finitely many equidistant functions in a natural way: the equidistant function one of whose focal sets is constituted by the pointwise minima of finitely many positive-valued continuous functions is given by the pointwise minima of the corresponding equidistant functions, see Figure 2.

\begin{figure}[h]
\label{fig:02}
\centering
\includegraphics[scale=0.3]{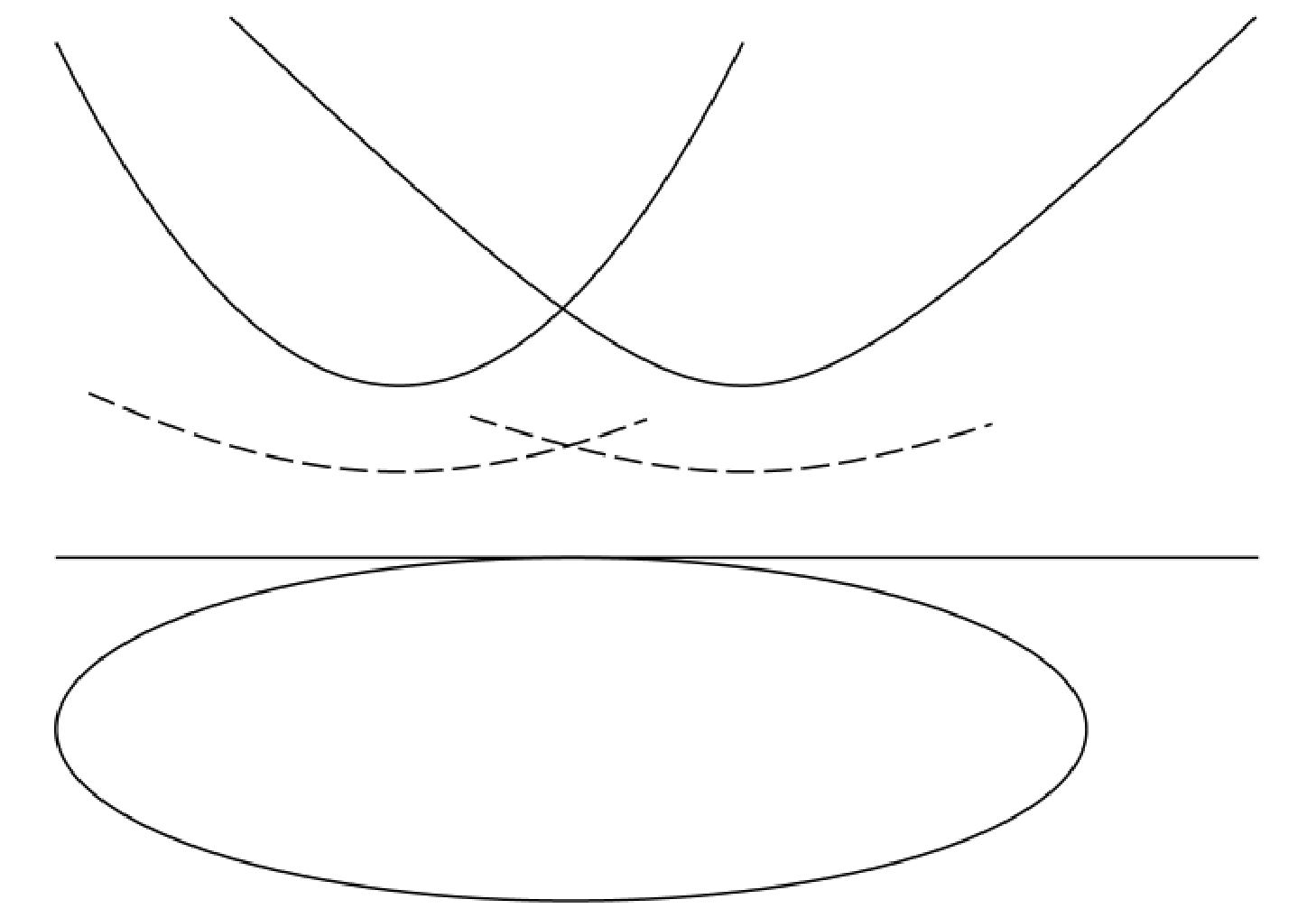}
\caption{The case of a min-function: see Theorem \ref{basic:01c}.}
\end{figure} 

In the second part of the paper we consider equidistant functions such that one of the focal sets is the epigraph of a convex function under some smoothness conditions. Independently of the dimension of the space we present a special parameterization for the equidistant points based on the closest point property of the epigraph as a convex set and we give the characterization of the equidistant functions as well. Illustrating how the formulas are working we present an example with a hyperboloid of revolution as one of the focal sets. 

Let $L\subset \mathbb{R}^n$ be a nonempty closed subset and for any $x\in \mathbb{R}^n$ let us introduce the set of closest points
$$L(x):=\{ y\in L\ | \ d(x,L)=d(x,y)\}\subset L.$$
  
\begin{lemma}
\label{lemma:first} If $L\subset \mathbb{R}^n$ is a nonempty closed subset, then  the mapping $h(x):=d(x,L)$ is Lipschitz continuous satisfying inequality
\begin{equation}
\label{ineq:first}
|h(x_2)-h(x_1)|\leq d(x_2, x_1)
\end{equation}
and equality
\begin{equation}
\label{eq:first}
h(x_2)-h(x_1)=d(x_2, x_1)
\end{equation}
holds if and only if $L(x_1)\subset L(x_2)$ and $x_2 - x_1 - L(x_1)$, that is $x_1$ is between $x_2$ and $y_1$ for any $y_1\in L(x_1)$ provided that $x_1\neq x_2$.
\end{lemma}

\begin{proof} Taking $y_1\in L(x_1)$, it follows that 
$$h(x_2)=d(x_2, L)\leq d(x_2, y_1)\leq d(x_2, x_1)+d(x_1, y_1)=$$
$$d(x_2, x_1)+d(x_1, L)=d(x_2, x_1)+h(x_1)$$
and, consequently, 
$$h(x_2)-h(x_1)\leq d(x_2, x_1).$$
Equality \eqref{eq:first} holds if and only if $y_1\in L(x_2)$ and $x_2 - x_1 - y_1$ for any $y_1\in L(x_1)$.
Changing the role of the points we have inequality \eqref{ineq:first}. 
\end{proof}

\section{General observations}
\label{basic} 

Consider the focal sets 
$$K := \{(t,0) \in \mathbb{R}^{n+1} \mid t \in \mathbb{R}^n \} \ \ \textrm{and}\ \ L := \{(t,f(t)) \in \mathbb{R}^{n+1} \mid t \in \mathbb{R}^n \},$$
where $f\colon \mathbb{R}^n\to \mathbb{R}^{+}$ is a positive-valued continuous function. First of all we clarify that the equidistant set $\{K=L\}$ can be given as the graph of a function $G\colon \mathbb{R}^n\to \mathbb {R}^+$. 

\begin{remark}
\label{remark:01}
Since $f$ is a positive-valued function, it is clear that there are no equidistant points above its graph and we can take its epigraph $L=\textrm{epi\ } f$ as a focal set in an equivalent way.
\end{remark}

\begin{lemma} 
\label{basic:01}
For any $x\in \mathbb{R}^n$, there is a uniquely determined positive real number $G(x)$ such that 
$$d((x,G(x)), L)=d((x,G(x)),K).$$
Inequalities $G(x)<y$ and $y<G(x)$ imply that 
$$d((x,y), L)< d((x,y),K) \quad \textrm{and} \quad d((x,y), K)< d((x,y),L),$$
respectively.  
\end{lemma}

\begin{proof}
Let $x\in \mathbb{R}^n$ be an arbitrary point in the space. Since the function $d((x, z), L)-z$
is positive at $z=0$ and negative at $z=f(x)$ it follows, by a continuity argument, that there is a positive real number $G(x)$ satisfying equation $d((x, G(x)), L)=G(x)=d((x, G(x)), K)$. It is uniquely determined because there are no equidistant points with $y\leq 0$ and equidistant points at different heights $0<y_1 < y_2$ give a contradiction as follows:
$$d((x, y_2), L)-d((x, y_1), L)=y_2 - y_1=d((x,y_2), (x,y_1)),$$
that is equality of type \eqref{eq:first} holds. Using Lemma \ref{lemma:first}, it follows that $L(x,y_1)\subset L(x,y_2)$ and $(x, y_2) - (x,y_1) - L(x, y_1)$. Since the points $(x, y_i)$ $(i=1, 2)$ determine a vertical line, we have an arrangement of the form  
$$(x, y_2) - (x,y_1) - (x, f(x))$$
but there are no equidistant points above the graph of the function $f$. Therefore it is impossible to find different equidistant points along a vertical line. In particular, by choosing $y_1=G(x)< y=y_2$, Lemma \ref{lemma:first} shows that 
$$d((x, y_2), L)-d((x, y_1), L) < d((x,y_2), (x,y_1))=y_2 - y_1\quad \Rightarrow \quad $$
$$d((x, y), L) < y=d((x,y), K)$$
because equality of type \eqref{eq:first} must be avoided. We can finish the proof by choosing $0<y_1=y < G(x)=y_2$ as follows:
$$d((x, y_2), L)-d((x, y_1), L) < d((x,y_2), (x,y_1))=y_2 - y_1\quad \Rightarrow \quad $$
$$d((x,y), K)=y < d((x, y), L)$$
and the inequality is automatically satisfied in case of $y\leq 0$. 
\end{proof}

\begin{definition} A function $G\colon \mathbb{R}^n\to \mathbb {R}^+$ is an equidistant function if its graph is the equidistant set of $K$ and $L$ for a positive-valued continuous function $f$.  
\end{definition}

\begin{lemma}
\label{basic:02} The equidistant function $G\colon \mathbb{R}^n\to \mathbb {R}^+$ belonging to a positive-valued continuous function $f\colon \mathbb{R}^n\to \mathbb{R}^+$ is continuous.  
\end{lemma}

\begin{proof} Taking $x_n\to x$, the sequence $y_n=G(x_n)$ is obviously bounded because of the continuity of the function $f$: for all but finitely many indices, 
$$0<y_n=G(x_n)< f(x_n) < f(x)+\varepsilon.$$
Therefore it has a convergent subsequence. Since the distance-measuring functions $d((x, y), L)$ and $d((x,y),K)$ are continuous, it can easily be seen that any convergent subsequence of $y_n$ gives a subsequence of $(x_n, y_n)$ tending to a point $(x,y)$ such that
$$d((x, y), L)=d((x,y),K).$$ 
The unicity of the equidistant point at $x$ implies that $y=G(x)$ is the common limit of the convergent subsequences of $y_n=G(x_n)$. Therefore $y_n=G(x_n)\to y=G(x)$. 
\end{proof}

\begin{theorem}
\label{basic:01c} Let $I$ be a finite nonempty index set and consider the family $\left\{f_i\mid i\in I\right\}$ of positive-valued continuous functions defined on $\mathbb{R}^n$ with corresponding equidistant functions $\left\{G_i\mid i\in I\right\}.$ 
If $G_{\min}$ is the equidistant function belonging to the function
	\[f_{\min}\colon \mathbb{R}^n\to\mathbb{R},\quad f_{\min}(x)=\min \left\{f_i(x)\mid i\in I\right\},
\]then
	\[G_{\min}(x)=\min \left\{G_i(x)\mid i\in I\right\}
\]for any $x\in\mathbb{R}^n$.
\end{theorem}

\begin{proof} 
Since all functions are positive-valued, we can use the epigraphs as focal sets in the sense of Remark \ref{remark:01}. Let $x\in \mathbb{R}^n$ be an arbitrary point, $L_i:=\textrm{epi\ } f_i$ $(i\in I)$ and suppose that 
$$y>\min \left\{G_i(x)\mid i\in I\right\}.$$
Then there exists at least one index $i\in I$ such that $y>G_i(x)$ and Lemma \ref{basic:01} implies that $d((x,y),L_i)<d((x,y),K)$. Since $L:=\textrm{epi\ } f_{\min}=\cup_{i\in I\ }\textrm{epi\ } f_i=\cup_{i\in I\ } L_i$, we have
	\[d((x,y),L)=d((x,y),\cup_{i\in I\ }L_i)\leq d((x,y),L_i)<d((x,y),K)
\]and Lemma \ref{basic:01} implies that $y>G_{\min}(x)$. On the other hand, suppose that 
$$y<\min \left\{G_i(x)\mid i\in I\right\}.$$
Then $y<G_i(x)$ for all $i\in I$ and Lemma \ref{basic:01} implies that $d((x,y),K)<d((x,y), L_i)$ for all $i\in I$. Thus  
	\[d((x,y),K)<\min \left\{d((x,y), L_i)\mid i\in I\right\}=d((x,y),\cup_{i\in I\ } L_i)=d((x,y), L)
\]and Lemma \ref{basic:01} implies that $y<G_{\min}(x)$. To sum up, $y>\min \left\{G_i(x)\mid i\in I\right\}$ implies that $y>G_{\min}(x)$ and $y<\min \left\{G_i(x)\mid i\in I\right\}$ implies that $y<G_{\min}(x)$. Therefore $G_{\min}(x)=\min \left\{G_i(x)\mid i\in I\right\}$.
\end{proof}

\begin{corollary}
\label{basic:03} The mapping $f\mapsto G$ is monotone in the sense that $f_1 < f_2$ implies that $G_1 < G_2$ for the pointwise ordering.
\end{corollary}

\begin{lemma}
\label{basic:04} The equidistant function $G\colon \mathbb{R}^n\to \mathbb {R}^+$ belonging to a convex function $f\colon \mathbb{R}^n\to \mathbb{R}^+$ is convex. 
\end{lemma}

\begin{proof} Since convexity implies continuity, the epigraph of the function $f$ is a closed convex subset. Therefore the distance-measuring function $d((x, y), L)$ is convex in its first variable, that is 
\begin{equation}
\label{eq:001}
d((x, y), L)\leq (1-\lambda)d((x_1, y_1), L)+\lambda d((x_2, y_2), L)=
\end{equation}
$$(1-\lambda)y_1+\lambda y_2=y=d((x,y),K),$$
where 
$$(x,y)=(1-\lambda)(x_1,y_1)+\lambda (x_2,y_2),$$ 
$G(x_1)=y_1$ and $G(x_2)=y_2$. Using Lemma \ref{basic:01}, inequality $y<G(x)\in \mathbb{R}^+$ implies that  
$$d((x, y), K) < d((x, y), L)$$
contradicting inequality \eqref{eq:001}. Therefore 
$$G(x)\leq y=(1-\lambda)y_1+\lambda y_2=(1-\lambda)G(x_1)+\lambda G(x_2)$$
as was to be proved. 
\end{proof}

\section{Equidistant functions belonging to positive-valued twice differentiable convex functions}

In what follows we are going to characterize the equidistant functions belonging to positive-valued twice continuously differentiable convex functions. As an important intermediate step, the equidistant points will be given in terms of special parametric expressions \eqref{eq:x} and \eqref{eq:y}. A detailed analysis of the parametric expressions results in necessary and sufficient conditions for a function $G\colon \mathbb{R}^n\to \mathbb {R}^+$ to be an equidistant function. Consider the focal sets 
$$K := \{(t,0) \in \mathbb{R}^{n+1} \mid t \in \mathbb{R}^n \} \ \ \textrm{and}\ \ L := \{(t,f(t)) \in \mathbb{R}^{n+1} \mid t \in \mathbb{R}^n \},$$
where $f\colon \mathbb{R}^n\to \mathbb{R}^{+}$ is a (positive-valued) twice continuously differentiable convex function.

\subsection{The parametric expressions of the equidistant points}
\label{parameterization}

Using the Euler-Monge parameterization $F(t)=(t, f(t))$ we have the normal vector field
$$\partial_1 F \times \ldots \times \partial_n F=\det \begin{bmatrix}
	e_1&\ldots & e_n&e_{n+1}\\
	1 & \ldots & 0&\partial_1 f\\
  %0 & 1 & \ldots & 0&\partial_2 f\\
	  \vdots & \vdots & \vdots & \vdots \\
  0 & \ldots & 1& \partial_n f
	\end{bmatrix}=
(-1)^{n+1}(\nabla f, -1),$$
where $e_1, \ldots, e_n, e_{n+1}$ is the canonical basis in $\mathbb{R}^{n+1}$ and $\nabla f$ denotes the gradient of the function $f$. The outer unit normal to the graph is
$$N=\frac{1}{\sqrt{1+|\nabla f|^2}}(\nabla f, -1).$$

\begin{theorem}
\label{parametric} The equidistant points can be written in the parametric form
\begin{equation}\label{eq:x}
  x(t) = t+\frac{f(t)}{1+ \sqrt{1+|\nabla f|^{2}(t)}}\nabla f(t)
	\end{equation} 
	and
	\begin{equation}
	\label{eq:y}
	y(t) = \frac{f(t)\sqrt{1+|\nabla f|^2(t)}}{1+ \sqrt{1+|\nabla f|^2(t)}},
 \end{equation}
where $t\in \mathbb{R}^n$ and $\nabla f$ denotes the gradient of the function $f$.
\end{theorem}

\begin{proof} The equidistant points are characterized by the equation $d((x,y), L)=y$. In other words
\begin{equation}
\label{eq:closest}
d((x, y), (t,f(t)))=y,
\end{equation}
where $(t, f(t))$ is the uniquely determined closest point of the convex epigraph to the point $(x,y)$. This means that the difference vector $(x,y)-(t,f(t))$ is proportional to the outer unit normal, that is
$$(x,y)=(t,f(t))+s N(t)$$
and equation \eqref{eq:closest} is equivalent to
$$(x,y)=(t,f(t))+y N(t).$$
In a more detailed form
$$x=t+\frac{y}{\sqrt{1+|\nabla f|^2(t)}}\nabla f(t)\ \ \textrm{and}\ \ y=f(t)-\frac{y}{\sqrt{1+|\nabla f|^2(t)}}.$$ 
The second equation allows us to express $y$ in terms of the parameter $t$ and we have equations \eqref{eq:x} and \eqref{eq:y}. 
\end{proof}

\begin{definition} The pair of the parametric expressions $x\colon \mathbb{R}^n\to \mathbb{R}^n$ and $y\colon \mathbb{R}^n\to \mathbb{R}^+$ is called the equidistant parameterization for the graph of the equidistant function.
\end{definition}

\begin{theorem}
\label{analysis} The mapping $x\colon \mathbb{R}^n\to \mathbb{R}^n$,
$$x(t) = t+\frac{f(t)}{1+ \sqrt{1+|\nabla f|^{2}(t)}}\nabla f(t)$$
is a one-to-one correspondence with nowhere vanishing Jacobian.
\end{theorem}

\begin{proof}
If $x_0\in \mathbb{R}^n$ is an arbitrary point in the space then $(x_0, G(x_0))$ is an equidistant point and $x(t_0)=x_0$, where $(t_0, f(t_0))$ is the closest point of the epigraph to $(x_0, G(x_0))$. To prove the one-to-one property suppose that $x(t_1)=x(t_2)=x_0$.
Since the closest point of the epigraph to $(x_0, G(x_0))$ is uniquely determined, it follows that $(t_1, f(t_1))=(t_2, f(t_2))$, that is $t_1=t_2$. The directional derivative along $v\in \mathbb{R}^n$ is 
\begin{gather}
\label{dirder:x}\partial_v x = v+\frac{\partial_v f\left(1+ \sqrt{1+|\nabla f|^{2}}\right)\sqrt{1+|\nabla f|^{2}}-f\langle \nabla f, \partial_v \nabla f\rangle}{\left(1+ \sqrt{1+|\nabla f|^{2}}\right)^2 \sqrt{1+|\nabla f|^{2}}}\nabla f+\\
\frac{f}{1+ \sqrt{1+|\nabla f|^{2}}}\partial_v \nabla f.\notag
\end{gather}
Taking the inner product with $\nabla f$ and $v\in \mathbb{R}^n$, respectively, we have
\begin{equation}
\label{innerprod:grad}
\langle \partial_v x, \nabla f \rangle= \partial_v f \left(1+\frac{| \nabla f |^2}{1+ \sqrt{1+|\nabla f|^{2}}}\right)+\frac{f \langle \nabla f, \partial_v \nabla f\rangle}{\sqrt{1+|\nabla f|^{2}}\left(1+ \sqrt{1+|\nabla f|^{2}}\right)}
\end{equation}
and
\begin{equation}
\label{innerprod:v}
\langle \partial_v x, v \rangle= |v|^2 +\frac{\left(\partial_v f\right)^2+f f''(v,v)}{1+ \sqrt{1+|\nabla f|^{2}}}-\frac{f \langle \nabla f, \partial_v \nabla f\rangle}{\left(1+ \sqrt{1+|\nabla f|^{2}}\right)^2\sqrt{1+|\nabla f|^{2}}}\partial_v f.
\end{equation}
Therefore
$$\langle \partial_v x, v+\frac{\partial_v f}{1+ \sqrt{1+|\nabla f|^{2}}} \nabla f\rangle=$$
$$=|v|^2 +\frac{\left(\partial_v f\right)^2+f f''(v,v)}{1+ \sqrt{1+|\nabla f|^{2}}}+\frac{\left(\partial_v f\right)^2}{1+ \sqrt{1+|\nabla f|^{2}}}\left(1+\frac{| \nabla f |^2}{1+ \sqrt{1+|\nabla f|^{2}}}\right)\geq |v|^2$$
because $f$ is a positive-valued convex function. The inequality shows that for any $t\in \mathbb{R}^n$, the mapping $x'(t)$ is regular:  $x'(t)(v)=\partial_v x (t)=0 \quad \Leftrightarrow \quad v=0.$
\end{proof}

\begin{corollary}
\label{specformG:01} Using the equidistant parameterization for the graph of the equidistant function, it can be given as $G=y\circ x^{-1}$.
\end{corollary}

In what follows we are going to express the function $f$ in terms of the equidistant parameterization for the graph of the equidistant function. Using equations \eqref{eq:x} and \eqref{eq:y}, 

\begin{equation}
\label{intro:g} g(t):=\frac{x(t)-t}{y(t)}=\frac{1}{\sqrt{1+|\nabla f|^2}}\nabla f \ \Rightarrow \ \frac{g}{\sqrt{1-|g|^2}}=\nabla f,
\end{equation}
where $|g(t)|< 1$ for any $t\in \mathbb{R}^n$. Since the convexity of $f$ is equivalent to the monotonicity of its gradient, it   follows that the mapping
$$X:=\frac{g}{\sqrt{1-|g|^2}}$$	
is a conservative vector field satisfying the monotonicity property $\langle X(t_1)-X(t_2), t_1-t_2\rangle\geq 0$ and its potential is  
\begin{equation}
\label{eq:pot}
f(t)\stackrel{\eqref{eq:x}, \eqref{eq:y}}{=}y(t)+\sqrt{y^2(t)-|x(t)-t|^2}=y(t)\left(1+\sqrt{1-|g(t)|^2}\right).
\end{equation}
We are going to formulate necessary and sufficient conditions for
\begin{equation}
\label{necandsuf:g}
\frac{g}{\sqrt{1-|g|^2}}=\nabla f, \ \ \textrm{where} \ \ f(t)=y(t)+\sqrt{y^2(t)-|x(t)-t|^2}.
\end{equation}

\begin{theorem}
\label{thm:grad}
Using the notation \eqref{intro:g}, equation \eqref{necandsuf:g} is equivalent to  
\begin{equation}
\label{grad:04}
\frac{\langle \partial_i x , g\rangle}{1+\sqrt{1-|g|^2}}=\partial_i y \qquad (i=1, \ldots,n). 
\end{equation}
\end{theorem}

\begin{proof}
In terms of the functions $x(t)$ and $y(t)$, equation \eqref{necandsuf:g} is equivalent to
$$\frac{x(t)-t}{\sqrt{y^2(t)-|x(t)-t|^2}}=\nabla f (t) \qquad (t\in \mathbb{R}^n),$$
\begin{equation}
\label{grad:01}
\frac{x^i(t)-t^i}{\sqrt{y^2(t)-|x(t)-t|^2}}=\partial_i f (t) \qquad (t\in \mathbb{R}^n, \ i=1, \ldots, n),
\end{equation}
where
$$\partial_i f (t)=\partial_i y (t)+\frac{y(t)\partial_i y (t)-\langle \partial_i x (t)-e_i, x(t)-t\rangle}{\sqrt{y^2(t)-|x(t)-t|^2}}$$
and $e_1$, $\ldots$, $e_n$ is the canonical basis in $\mathbb{R}^n$. Therefore equation \eqref{grad:01} is equivalent to 
$$
 0=\partial_i y (t)+\frac{y(t)\partial_i y (t)-\langle \partial_i x (t), x(t)-t\rangle}{\sqrt{y^2(t)-|x(t)-t|^2}} \qquad (t\in \mathbb{R}^n, \ i=1, \ldots, n),
$$
\begin{equation}
\label{grad:03}
 \langle \partial_i x (t), x(t)-t\rangle=y(t)\partial_i y(t)+\sqrt{y^2(t)-|x(t)-t|^2}\partial_i y (t) \qquad (t\in \mathbb{R}^n, \ i=1, \ldots, n),
\end{equation}
and we have equation \eqref{grad:04} by dividing with
$$y(t)+\sqrt{y^2(t)-|x(t)-t|^2}=y(t)\left(1+\sqrt{1-|g(t)|^2}\right).$$
\end{proof}

\begin{theorem} 
\label{charac:param}
Let $x\colon \mathbb{R}^n\to \mathbb{R}^n$ and $y\colon \mathbb{R}^n\to \mathbb{R}^+$ be continuously differentiable functions and suppose that the norm of
$$g(t)=\frac{x(t)-t}{y(t)} \quad (t\in \mathbb{R}^n)$$ 
is less than one, the mapping 
$$X:=\frac{g}{\sqrt{1-|g|^2}}$$	
satisfies the monotonicity property $\langle X(t_1)-X(t_2), t_1-t_2\rangle\geq 0$ and the partial derivatives of $x$ and $y$ are related as
$$\frac{\langle \partial_i x , g\rangle}{1+\sqrt{1-|g|^2}}=\partial_i y \qquad (i=1, \ldots,n).$$
Then the functions $x$ and $y$ give the equidistant parameterization for the graph of the equidistant function belonging to
\begin{equation}
\label{def:f}
f(t)=y(t)\left(1+\sqrt{1-|g(t)|^2}\right)=y(t)+\sqrt{y^2(t)-|x(t)-t|^2}.
\end{equation}
\end{theorem}

\begin{proof} Let us introduce the function $f$ by the formula \eqref{def:f}. Using Theorem \ref{thm:grad}, we can reproduce formula \eqref{necandsuf:g} as
\begin{equation}
\label{def:grad} \frac{g}{\sqrt{1-|g|^2}}=\nabla f.
\end{equation}
Since the functions $x$ and $y$ appear in the expression of $\nabla f$ without their derivatives, it follows that $f$ is a twice continuously differentiable function. It is also convex because of the monotonicity property. Using \eqref{def:f} and \eqref{def:grad} we can reproduce formulas \eqref{eq:x} and \eqref{eq:y} as follows: 
$$1+|\nabla f|^2=1+\frac{|g|^2}{1-|g|^2} \ \Rightarrow\ \sqrt{1+|\nabla f|^2}=\frac{1}{\sqrt{1-|g|^2}}$$
and
$$1+\sqrt{1+|\nabla f|^2}=\frac{1+\sqrt{1-|g|^2}}{\sqrt{1-|g|^2}}.$$
Therefore
$$\frac{f(t)}{1+ \sqrt{1+|\nabla f|^{2}(t)}}\nabla f(t)=\frac{f(t) \sqrt{1-|g(t)|^2}}{1+\sqrt{1-|g(t)|^2}}\frac{g(t)}{\sqrt{1-|g(t)|^2}}=y(t)g(t)=x(t)-t$$
and
$$\frac{f(t)\sqrt{1+|\nabla f|^2(t)}}{1+ \sqrt{1+|\nabla f|^2}(t)}=\frac{f(t)}{1+\sqrt{1-|g(t)|^2}}=y(t).$$ 
\end{proof}

\subsection{The characterization of equidistant functions}
\label{characterization}

In what follows we are going to investigate the equidistant function of the form
\begin{equation}
\label{intro:G}
y(t) = G(x(t)),\ \ \textrm{where}\ \ G\colon \mathbb{R}^n \mapsto \mathbb{R}^+.
\end{equation}
It is a convex (Lemma \ref{basic:04}), continuously differentiable (Corollary \ref{specformG:01}) and positive-valued function. Differentiating equation \eqref{intro:G},
$$\langle \partial_i x, \nabla G \circ x \rangle=\partial_i y \qquad (i=1, \ldots, n).$$
By Theorem  \ref{analysis}, the partial derivatives of the mapping $x$ form a basis. Using Theorem \ref{thm:grad}, 
\begin{equation}
\label{eq:key0}
 \nabla G\circ x(t)=\frac{g(t)}{1+\sqrt{1-|g(t)|^2}}=\frac{x(t)-t}{y(t)+\sqrt{y^2(t)-|x(t)-t|^2}}=
\end{equation}
$$\frac{x(t)-t}{G\circ x(t)+\sqrt{G^2\circ x(t)-|x(t)-t|^2}}.$$
In particular, by Theorem \ref{charac:param},
\begin{equation}
\label{normG:01} |\nabla G| \circ x(t)=\frac{|g(t)|}{1+\sqrt{1-|g(t)|^2}}<1 \quad (t\in \mathbb{R}^n).
\end{equation}
Since the functions $x$ and $y$ appear in the expression of $\nabla G$ without their derivatives, $G$ is a twice continuously differentiable function. By some further straightforward computations,
$$\sqrt{G^2\circ x(t)-|x(t)-t|^2} \nabla G\circ x(t)=\left(x(t)-t\right) -G\circ x(t) \nabla G\circ x(t).$$
Taking the norm square of both sides,
$$\left(G^2\circ x(t)-|x(t)-t|^2\right) |\nabla G|^2\circ x(t)=$$
$$|x(t)-t|^2 -2G\circ x(t) \langle x(t)-t, \nabla G \circ x(t) \rangle+G^2\circ x(t) |\nabla G|^2 \circ x(t).$$
Subtracting the common terms, we can compute the Fourier coefficient of $\nabla G\circ x (t)$ with respect to the difference vector $x(t)-t$ as follows:
$$-|x(t)-t|^2 |\nabla G|^2\circ x(t)=|x(t)-t|^2 -2G\circ x(t) \langle x(t)-t, \nabla G \circ x(t) \rangle$$
and, consequently, 
$$
 \frac{\langle x(t)-t, \nabla G \circ x(t) \rangle}{|x(t)-t|^2}=\frac{1+|\nabla G|^2\circ x(t)}{2G\circ x(t)} \ \Rightarrow \ \nabla G\circ x(t)=\frac{1+|\nabla G|^2\circ x(t)}{2G\circ x(t)} \left(x(t)-t\right)
$$
provided that $x(t)\neq t$. Therefore
\begin{equation}
\label{eq:key}
 x(t)-t=\frac{2G\circ x(t)}{1+|\nabla G|^2\circ x(t)} \nabla G\circ x(t) \ \Rightarrow\ t=x(t)-\frac{2G\circ x(t)}{1+|\nabla G|^2\circ x(t)} \nabla G\circ x(t).
\end{equation}
Formula \eqref{eq:key} also holds in the case of $x(t)=t$ because the vanishing of the difference vector $x(t)-t$ implies that $\nabla G\circ x(t)=0$ in the sense of \eqref{eq:key0}. Introducing the mapping
\begin{equation}
\label{eq:defH}
H(x)=x-\frac{2G(x)}{1+|\nabla G (x)|^2} \nabla G(x) \quad \Rightarrow \quad H\circ x(t)=t,
\end{equation}
we can express the functions $x$ and $y$ in terms of $G$ as follows: 
\begin{equation}                
\label{eq:equidf} x(t)=H^{-1}(t) \quad \textrm{and}\quad  y(t)=G\circ x (t)
\end{equation}
for any $t\in \mathbb{R}^n$. Using Theorem \ref{charac:param}, the following result can be formulated. 

\begin{theorem}
\label{equidfunc:flat}
A twice continuously differentiable, convex and positive-valued function $G\colon \mathbb{R}^n \to \mathbb{R}^+$ is an equidistant function if and only if 
\begin{equation}
\label{def:H}
H(x) := x - \frac{2 G(x)}{1+|\nabla G(x)|^2} \nabla G(x)
\end{equation}
is a one-to-one correspondence with nowhere vanishing Jacobian, $|\nabla G(x)|<1$ for any $x\in \mathbb{R}^n$ and the mapping
$$Y=\frac{2\nabla G}{1-|\nabla G|^2}$$
satisfies the monotonicity property 
\begin{equation}
\label{monotonicity}
\langle Y(x_1)-Y(x_2), H(x_1)-H(x_2)\rangle\geq 0.
\end{equation}
\end{theorem}

\begin{proof} Let us introduce the parameterization $x(t):=H^{-1}(t)$ and $y(t)=G\circ x(t)$ for the points of the graph of $G$. We are going to prove that it is an equidistant parameterization by checking the conditions in Theorem \ref{charac:param}. By substituting $x=x(t)$ in formula \eqref{def:H},  
$$g(t)=\frac{x(t)-t}{y(t)}=\frac{2\nabla G}{1+|\nabla G|^2} \circ x(t).$$
Its norm is less than one because $|\nabla G(x)|<1$ for any $x\in \mathbb{R}^n$. On the other hand
$$\sqrt{1-|g|^2}=\frac{1-|\nabla G|^2}{1+|\nabla G|^2}\circ x \quad \Rightarrow \quad X=\frac{g}{\sqrt{1-|g|^2}}=\frac{2 \nabla G}{1-|\nabla G|^2} \circ x=Y\circ x$$
satisfies the monotonicity property by substituting $x_1=x(t_1)$ and $x_2=x(t_2)$ in formula \eqref{monotonicity}. As a straightforward computation shows,
$$1+\sqrt{1-|g|^2}=\frac{2}{1+|\nabla G|^2}\circ x \quad \Rightarrow \quad \frac{g}{1+\sqrt{1-|g|^2}}=\nabla G\circ x$$
and the relation
$$\frac{\langle \partial_i x , g\rangle}{1+\sqrt{1-|g|^2}}=\partial_i y \qquad (i=1, \ldots,n)$$
is automatically satisfied because of $y=G\circ x$. Therefore the functions $x$ and $y$ give the equidistant parameterization for the graph of the equidistant function $G$ belonging to 
$$f(t)=y(t)\left(1+\sqrt{1-|g(t)|^2}\right)=y(t)+\sqrt{y^2(t)-|x(t)-t|^2}.$$
The converse statement follows by the results and computations in Subsections \ref{parameterization} and \ref{characterization}. 
\end{proof}

\begin{corollary}
\label{envelope:parabolas} The graph of the equidistant function $G$ is the envelope of the parametric family of paraboloids given by
$$y=\frac{|x-t|^2}{2f(t)}+\frac{f(t)}{2} \qquad (t\in \mathbb{R}^n).$$
\end{corollary}

\begin{proof} Using the parametric expressions \eqref{eq:x} and \eqref{eq:y}, a straightforward calculation shows that
$$y(t)=\frac{|x(t)-t|^2}{2f(t)}+\frac{f(t)}{2}$$
for any $t\in \mathbb{R}^n$. On the other hand, 
$$\left(\frac{x(t)-t}{f(t)}, -1\right)=\left(\nabla G \circ x(t),-1\right),$$
that is the gradient vector field of the paraboloid at the contact point is the outer normal vector field of the graph of the equidistant function $G$ as 
$$f(t)=y(t)+\sqrt{y^2(t)-|x(t)-t|^2}$$
(formula \eqref{eq:pot}) and 
$$\nabla G\circ x(t)=\frac{g(t)}{1+\sqrt{1-|g(t)|^2}}=\frac{x(t)-t}{y(t)+\sqrt{y^2(t)-|x(t)-t|^2}}$$
(formula \eqref{eq:key0}) show. Since the focal point is running along the graph of the function $f$ we can also conclude that any tangent hyperplane of the equidistant function $G$ bounds all the paraboloids because the points under the tangent hyperplanes of $G$ are closer to the hyperplane $K=\mathbb{R}^n\times \{0\}\subset \mathbb{R}^{n+1}$ than to the epigraph of the function $f$. 
\end{proof}

\section{An example}

\subsection{Univariate functions} Consider the very special case of dimension $n=1$ to find the equidistant function corresponding to a function $f$ of one variable. In particular, $\nabla f=f'$. Let's choose for example the function $f\colon \mathbb{R}\to\mathbb{R}$, $f(t)=\sqrt{t^2+1}$. Using formulas \eqref{eq:x} and \eqref{eq:y},
	\[x(t)=t+\frac{t}{1+\sqrt{1+\frac{t^2}{t^2+1}}}=t+\frac{t\sqrt{t^2+1}}{\sqrt{t^2+1}+\sqrt{2t^2+1}}
\]
	\[y(t)=\frac{\sqrt{t^2+1}\sqrt{1+\frac{t^2}{t^2+1}}}{1+\sqrt{1+\frac{t^2}{t^2+1}}}=\frac{\sqrt{t^2+1}\sqrt{2t^2+1}}{\sqrt{t^2+1}+\sqrt{2t^2+1}}
\]To find the inverse of the function $x\colon \mathbb{R}\to\mathbb{R}$ we solve the following parametric equation for $t$:
\begin{equation}
\label{cubic}
x=t+\frac{t\sqrt{t^2+1}}{\sqrt{t^2+1}+\sqrt{2t^2+1}}
\end{equation}Then
	\[(x-t)\left(\sqrt{t^2+1}+\sqrt{2t^2+1}\right)=t\sqrt{t^2+1}
\]
	\[(x-t)\sqrt{2t^2+1}=(2t-x)\sqrt{t^2+1}
\]
	\[(x-t)^2(2t^2+1)=(2t-x)^2(t^2+1)
\]
	\[-2t^4+x^2t^2-3t^2+2xt=0
\]
	\[-2t\left(t^3-\frac{x^2-3}{2}\,t-x\right)=0
\]The solution $t=0$ corresponds to $x(0)=0$. Let's focus on the cubic equation
\begin{equation}
\label{cubic01} 
t^3-\frac{x^2-3}{2}\,t-x=0.
\end{equation}We know from the theory of the cubic equations that if
\begin{equation}
\label{cubic02} 
\frac{x^2}{4}-\frac{(x^2-3)^3}{216}>0
\end{equation}then the only real solution of the equation \eqref{cubic01} for $t$ is
\begin{equation}
\label{cubic03} 
t=\sqrt[3]{\frac{x}{2}+\sqrt{\frac{x^2}{4}-\frac{(x^2-3)^3}{216}}}+\sqrt[3]{\frac{x}{2}-\sqrt{\frac{x^2}{4}-\frac{(x^2-3)^3}{216}}}.
\end{equation}A complete analysis shows that the inequality \eqref{cubic02} is satisfied exactly when $x_1<x<x_2$, where
	\[x_1=-\sqrt{3} \sqrt{\sqrt[3]{4}+\sqrt[3]{2}+1}\ \ \textrm{and}\ \ x_2=\sqrt{3} \sqrt{\sqrt[3]{4}+\sqrt[3]{2}+1}
\] By applying a simple continuity argument, substituting $x_1$ or $x_2$ for $x$ in the formula \eqref{cubic03} yields a solution of \eqref{cubic01}. Hence the formula \eqref{cubic03} remains valid over the closed interval $\left[x_1,x_2\right]$. In what follows we are looking for the roots in case of $x<x_1$ or $x_2<x$. It's easy to see that, if $x>0$, then the equation \eqref{cubic} can only have positive solutions, while if $x<0$, then the equation \eqref{cubic} can only have negative solutions for $t$. Therefore we are looking for the positive real root of the equation \eqref{cubic01} in case of 
$x_2< x$ and the root we are looking for is negative in case of $x<x_1$.
If $x<x_1$ or $x_2<x$, then 
	\[\frac{x^2}{4}-\frac{(x^2-3)^3}{216}<0
\]and the two complex square roots of the left hand side are
	\[\pm i\,\sqrt{-\frac{x^2}{4}+\frac{(x^2-3)^3}{216}},
\]where $\sqrt{\phantom{1}}$ denotes the real square root function defined on the set of positive real numbers. Let's define the following complex numbers
	\[z_1=\frac{x}{2}+i\,\sqrt{-\frac{x^2}{4}+\frac{(x^2-3)^3}{216}}\ \ \textrm{and}\ \ z_2=\frac{x}{2}-i\,\sqrt{-\frac{x^2}{4}+\frac{(x^2-3)^3}{216}}.
\]Let the trigonometric form of $z_1$ and $z_2$ be
	\[z_1=\sqrt{\frac{(x^2-3)^3}{216}}\,\big(\cos(\alpha_1)+i\,\sin(\alpha_1)\big)\ \ \textrm{and}\ \ z_2=\sqrt{\frac{(x^2-3)^3}{216}}\,\big(\cos(\alpha_2)+i\,\sin(\alpha_2)\big).
\]The complex third roots of $z_j$ are
	\[\left(\sqrt[3]{z_j}\right)_{k+1}=\sqrt{\frac{x^2-3}{6}}\,\left(\cos\left(\frac{1}{3}\alpha_j+\frac{2k\pi}{3}\right)+i\,\sin\left(\frac{1}{3}\alpha_j+\frac{2k\pi}{3}\right)\right)
\]for $j\in\left\{1,2\right\}$ and $k\in\left\{0,1,2\right\}$. Now the equation \eqref{cubic01} has three real roots:
	\[t_1=\left(\sqrt[3]{z_1}\right)_1+\left(\sqrt[3]{z_2}\right)_1,\quad t_2=\left(\sqrt[3]{z_1}\right)_2+\left(\sqrt[3]{z_2}\right)_2,\quad t_3=\left(\sqrt[3]{z_1}\right)_3+\left(\sqrt[3]{z_2}\right)_3.
\]If $x_2<x$, then
	\[\alpha_{1,2}=\pm\arctan\left(\sqrt{-1+\frac{4}{\ x^2}\,\frac{(x^2-3)^3}{216}}\right)=\pm\arctan\left(\sqrt{\frac{(x^2-3)^3}{54x^2}-1}\right).
\] Since
$0<\alpha_1<\frac{\pi}{2}$, the only positive root of \eqref{cubic01} is
\begin{equation}
\label{cubict1} 
t_1=\left(\sqrt[3]{z_1}\right)_1+\left(\sqrt[3]{z_2}\right)_1=2\,\sqrt{\frac{x^2-3}{6}}\,\cos\left(\frac{1}{3}\,\arctan\left(\sqrt{\frac{(x^2-3)^3}{54x^2}-1}\right)\right).
\end{equation}If $x<x_1$, then
	\[\alpha_{1,2}=\pi\mp\arctan\left(\sqrt{-1+\frac{4}{\ x^2}\,\frac{(x^2-3)^3}{216}}\right)=\pi\mp\arctan\left(\sqrt{\frac{(x^2-3)^3}{54x^2}-1}\right).
\] Since $\frac{\pi}{2}< \alpha_1\leq \pi$, the only negative root of \eqref{cubic01} is
\begin{equation}
\label{cubict2}
t_2=\left(\sqrt[3]{z_1}\right)_2+\left(\sqrt[3]{z_2}\right)_2=-2\,\sqrt{\frac{x^2-3}{6}}\,\cos\left(\frac{1}{3}\,\arctan\left(\sqrt{\frac{(x^2-3)^3}{54x^2}-1}\right)\right)
\end{equation}
This finally leads to the following formula for the inverse function of $x\colon \mathbb{R}\to\mathbb{R}$:
\begin{equation}
\label{cubicfinal} 
x^{-1}(s)=\begin{cases}
	-2\,\sqrt{\frac{s^2-3}{6}}\,\cos\left(\frac{1}{3}\,\arctan\left(\sqrt{\frac{(s^2-3)^3}{54s^2}-1}\right)\right), & \textrm{ if } s<x_1,\\
	\\
	\ \ \sqrt[3]{\frac{s}{2}+\sqrt{\frac{s^2}{4}-\frac{(s^2-3)^3}{216}}}+\sqrt[3]{\frac{s}{2}-\sqrt{\frac{s^2}{4}-\frac{(s^2-3)^3}{216}}}, & \textrm{ if } x_1\leq s\leq x_2,\\
	\\
	\ \ 2\,\sqrt{\frac{s^2-3}{6}}\,\cos\left(\frac{1}{3}\,\arctan\left(\sqrt{\frac{(s^2-3)^3}{54s^2}-1}\right)\right), & \textrm{ if } x_2<s.
	\end{cases}
\end{equation}Hence the equidistant function corresponding to $f\colon \mathbb{R}\to\mathbb{R}$, $f(t)=\sqrt{t^2+1}$ is
	\[G\colon \mathbb{R}\to\mathbb{R},\quad G(s)=y\circ x^{-1}(s)=\frac{\sqrt{\left(x^{-1}(s)\right)^2+1}\sqrt{2\left(x^{-1}(s)\right)^2+1}}{\sqrt{\left(x^{-1}(s)\right)^2+1}+\sqrt{2\left(x^{-1}(s)\right)^2+1}},
\]where $x^{-1}(s)$ is defined by formula \eqref{cubicfinal}.

\begin{figure}
\centering     
\includegraphics[scale=0.3]{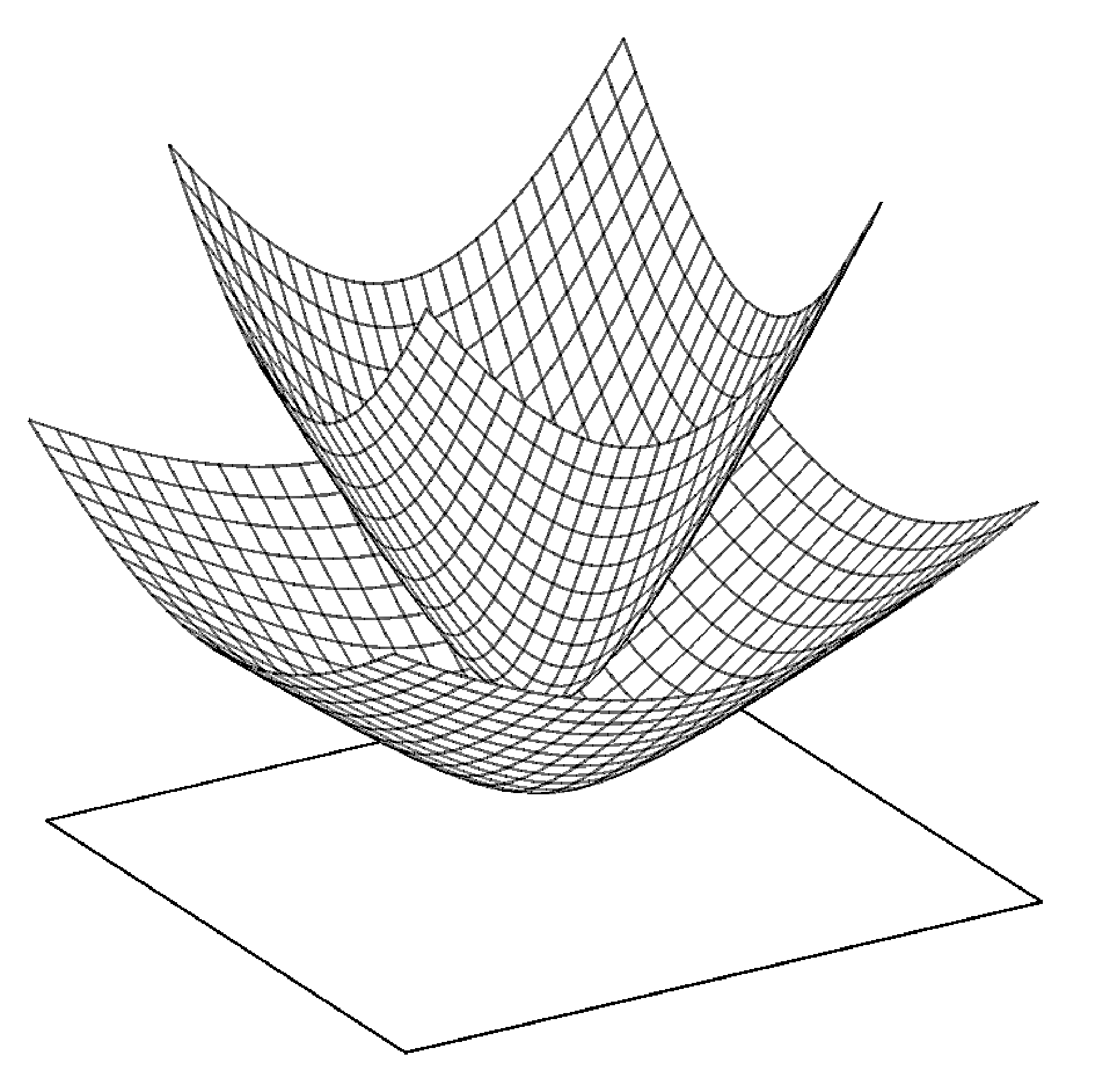}
\caption{The function $f\colon \mathbb{R}^2\to\mathbb{R}$, $f(t)=\sqrt{|t|^2+1}$ and the corresponding equidistant function.}
\end{figure}

\subsection{Multivariate functions} Now let $n\in\mathbb{N}$ be a positive integer $(n\geq 2)$ and consider the function $f\colon \mathbb{R}^n\to\mathbb{R}$, $f(t)=\sqrt{|t|^2+1}$. Then
	\[\nabla f(t)=\frac{t}{\sqrt{|t|^2+1}},\quad f(t) \nabla f(t)=t. 
\] Using formulas \eqref{eq:x} and \eqref{eq:y}, 
	\[x(t)=t+\frac{t}{1+\sqrt{1+\frac{|t|^2}{|t|^2+1}}}=t+\frac{t\sqrt{|t|^2+1}}{\sqrt{|t|^2+1}+\sqrt{2|t|^2+1}}
\]
	\[y(t)=\frac{\sqrt{|t|^2+1}\sqrt{1+\frac{|t|^2}{|t|^2+1}}}{1+\sqrt{1+\frac{|t|^2}{|t|^2+1}}}=\frac{\sqrt{|t|^2+1}\sqrt{2|t|^2+1}}{\sqrt{|t|^2+1}+\sqrt{2|t|^2+1}}
\]To find the inverse of the function $x\colon \mathbb{R}^n\to\mathbb{R}^n$ we solve the following equation for $t$:
\begin{equation}
\label{cubicgen}
x=\left(1+\frac{\sqrt{|t|^2+1}}{\sqrt{|t|^2+1}+\sqrt{2|t|^2+1}}\right)t
\end{equation}Here the coefficient of $t$ is strictly positive, thus $t$ and $x$ have the same direction. Furthermore, if $x=0$ then $t=0$. Otherwise it's enough to know how $|t|$ depends on $x$ to find a formula of the inverse function $x^{-1}\colon \mathbb{R}^n\to\mathbb{R}^n$. The equation \eqref{cubicgen} gives
	\[|x|=|t|+\frac{|t| \sqrt{|t|^2+1}}{\sqrt{|t|^2+1}+\sqrt{2|t|^2+1}}.
\]Then we can apply a similar argumentation as in the case of equation \eqref{cubic} to finally get
	\begin{equation}
\label{cubicgenfinal} 
x^{-1}(s)=\begin{cases}
	\ \ \left(\sqrt[3]{\frac{|s|}{2}+\sqrt{\frac{|s|^2}{4}-\frac{(|s|^2-3)^3}{216}}}+\sqrt[3]{\frac{|s|}{2}-\sqrt{\frac{|s|^2}{4}-\frac{(|s|^2-3)^3}{216}}}\right) \frac{s}{|s|}, & \textrm{ if } 0< |s| \leq x_2,\\
	\\
	\ \ 0, & \ \textrm{if } s=0,\\
	\\
	\ \ 2\,\sqrt{\frac{|s|^2-3}{6}}\,\cos\left(\frac{1}{3}\,\arctan\left(\sqrt{\frac{(|s|^2-3)^3}{54|s|^2}-1}\right)\right) \frac{s}{|s|}, & \textrm{ if } x_2<|s|,
	\end{cases}
\end{equation}where $x_2=\sqrt{3} \sqrt{\sqrt[3]{4}+\sqrt[3]{2}+1}$. Then the equidistant function corresponding to $f\colon \mathbb{R}^n\to\mathbb{R}$, $f(t)=\sqrt{|t|^2+1}$ is
	\[G\colon \mathbb{R}^n\to\mathbb{R},\quad G(s)=y\circ x^{-1}(s)=\frac{\sqrt{\left|x^{-1}(s)\right|^2+1}\sqrt{2\left|x^{-1}(s)\right|^2+1}}{\sqrt{\left|x^{-1}(s)\right|^2+1}+\sqrt{2\left|x^{-1}(s)\right|^2+1}},
\]where $x^{-1}(s)$ is defined by formula \eqref{cubicgenfinal}.

\section{Acknowledgement}

Myroslav Stoika is supported by the Visegrad Scholarship Program. M\'ark Ol\'ah has received funding from the HUN-REN Hungarian Research Network.

\end{document}